\documentclass[11pt]{amsart}
\usepackage{amsmath}
\usepackage{amsfonts}

\usepackage{amscd}
\usepackage{amsthm}
\usepackage{amssymb} \usepackage{latexsym}
\usepackage{eufrak}
\usepackage{euscript}
\usepackage{epsfig}
\usepackage{graphics}
\usepackage{array}
\usepackage{enumerate}
\usepackage{dsfont}
\usepackage{color}
\usepackage{wasysym}

\usepackage{pdfsync}
\usepackage[colorlinks,citecolor=red,pagebackref,hypertexnames=false, breaklinks]{hyperref}

\calclayout
\allowdisplaybreaks

 \newcommand{\abs}[1]{{\left\lvert{#1}\right\rvert}}
\newcommand{\norm}[1]{{\left\lVert{#1}\right\rVert}}
\newcommand{\br}[1]{{\left({#1}\right)}}

\newcommand{\bel}[1]{\begin{equation}\label{#1}}

\newcommand{\be}{\begin{equation}}

\newcommand{\ba}{\begin{eqnarray}}
\newcommand{\ea}{\end{eqnarray}}

\newcommand{\qe}{\end{equation}}
\newcommand{\R}{{\mathbb R}}
\newcommand{\N}{{\mathbb N}}
\newcommand{\Z}{{\mathbb Z}}

\newcommand{\vol}{{\mathrm{vol}}}

\newcommand{\wt}{\widetilde}
\newcommand{\HH}{\mathcal{H}}
\newcommand{\es}{\mathrm{ess}}

\newcommand{\Hmm}[1]{\leavevmode{\marginpar{\tiny%
$\hbox to 0mm{\hspace*{-0.5mm}$\leftarrow$\hss}%
\vcenter{\vrule depth 0.1mm height 0.1mm width \the\marginparwidth}%
\hbox to
0mm{\hss$\rightarrow$\hspace*{-0.5mm}}$\\\relax\raggedright #1}}}

\theoremstyle{plain}
\newtheorem{thm}{Theorem}[section]
\newtheorem{lemma}[thm]{Lemma}

\newtheorem{prop}[thm]{Proposition}

\theoremstyle{definition}

\newtheorem{rem}[thm]{Remark}

\newtheorem{example}[thm]{Example}

\numberwithin{equation}{section}

\begin{document}

\date{\today}
\title{Riesz transforms for bounded Laplacians on graphs}

\author{Li Chen}
\address{Li Chen, Department of Mathematics, University of Connecticut,
341 Mansfield Road Storrs, CT  06269-1009, USA} 
\email{li.4.chen@uconn.edu}

\author{Thierry Coulhon}
\address{Thierry Coulhon,  PSL  Research University,  F-75006 Paris, France}
\email{thierry.coulhon@icloud.com}

\author{Bobo Hua}
\address{Bobo Hua, School of Mathematical Sciences, LMNS, Fudan University, Shanghai 200433, China}
\email{bobohua@fudan.edu.cn}
\thanks{Li Chen has been supported in part by ICMAT Severo Ochoa project SEV-2011-0087 and she acknowledges that the research leading to these results has received funding from the European Research Council under the European Union's Seventh Framework Programme (FP7/2007-2013)/ ERC agreement no. 615112 HAPDEGMT. Bobo Hua is supported by NSFC, grant no. 11401106.}

\begin{abstract}
%


We study several problems related to the $\ell^p$ boundedness of  Riesz transforms  for graphs endowed with so-called bounded Laplacians. Introducing  a proper notion of gradient of functions on edges, we prove for $p\in(1,2]$  an $\ell^p$  estimate for the gradient of the continuous time heat semigroup, an $\ell^p$ interpolation inequality as well as the $\ell^p$ boundedness of the  modified Littlewood-Paley-Stein functions for all graphs with bounded Laplacians. This yields an analogue to Dungey's results in \cite{Dungey08} while removing some additional assumptions. Coming back to the classical notion of gradient, we give a counterexample to the interpolation inequality hence to the boundedness of Riesz transforms for bounded Laplacians for $1<p<2$. Finally, we prove the boundedness of the Riesz transform for $1< p<\infty$ under the assumption of positive spectral gap.

 \end{abstract}
\maketitle
\tableofcontents


\section{Introduction}

On the Euclidean space $\R^n$, as is well-known, the Riesz transform $\nabla \Delta^{-1/2}$ is bounded on $L^p(\R^n)$ for any $1<p<\infty$, (here we adopt a convention that the Laplacian $\Delta$ is a non-negative operator on $L^2(\R^n)$).
On a general Riemannian manifold $M$ without boundary, whether the Riesz transform $\nabla \Delta_M^{-1/2}$ is bounded on $L^p(M,\vol),$ i.e.
\begin{equation}\label{eq:riesz manifold}
\||\nabla f|\|_{L^p(M,\vol)}\leq C\|\Delta_M^{1/2}f\|_{L^p(M,\vol)},\quad\forall f\in {\mathcal C}_0^{\infty}(M),
\end{equation}
encodes much geometric information of $M$. The boundedness of Riesz transforms on manifolds has been investigated thoroughly in the literature, see \cite{Strichartz83}, \cite{Bakry87} and the references in  \cite{AuscherCoulhonDuongHofmann04}. The results turn out to be quite different for  cases $1<p\leq 2$ and $p>2,$ see \cite{CoulhonDuong99,CoulhonDuong03,  AuscherCoulhonDuongHofmann04, AuscherCoulhon05,ChenCoulhonFeneuilRuss15}.  In the present paper, we concentrate on purely discrete phenomena.

The Riesz transform for  the so-called normalized Laplacians on graphs (see definition below) has  been studied by many authors, see for example \cite{Russ00,Russ01,Dungey04,Dungey08,BadrRuss09} etc.
In the present paper, we prove some results on Riesz transforms and Littlewood-Paley functions for a more general class of Laplacians on graphs, called bounded Laplacians. Without loss of generality, all manifolds and graphs we consider are connected.

Let us now introduce the setting  and then state our main results.


\subsection{Setting}

Let $(V,E)$ be an undirected combinatorial graph with the set of vertices $V$ and
the set of edges $E$, a symmetric subset of $V\times V$. Two vertices $x,y\in V$ are called neighbors  if they are connected by an edge $\{x,y\}\in E$, which is denoted by $x\sim y$.
We say that $(V,E)$ is locally finite if for each vertex there are only finitely many neighbors. The graph is called connected if for any distinct $x,y\in V$ there is a finite sequence of vertices $\{x_i\}_{i=0}^n,$ $n\in \N,$ such that $x=x_0\sim x_1\sim \cdots\sim x_n=y.$ In this paper, we only consider connected, locally finite graphs.

On $(V,E),$ we assign weights on vertices and edges as follows:
$$
\nu:V\to \R_{+},\quad V\ni x\mapsto \nu_x
$$
and
$$
\mu:E\to \R_{+},\quad E\ni\{x,y\}\mapsto \mu_{xy}=\mu_{yx}
$$
and call the quadruple $G=(V,E,\nu,\mu)$ a \emph{weighted graph}. A priori, there is no relation between the weights $\mu$ and $\nu$, which is similar to the setting of weighted manifolds (or manifolds with density), see \cite{Mor05,WW09,Gri09}. Later we will introduce a condition linking them via the boundedness of the Laplacians, see \eqref{e:bounded Laplacian}.

 Let $(V,\nu)$ and $(E,\mu)$ denote the discrete measure spaces on $V$ and $E$ equipped with the measures $\nu$ and $\mu$ respectively. We denote by $\R^V$ and $\R^E$ the set of real-valued functions on $V$ and $E,$ and by ${\mathcal C}_0(V)$ the set of finitely supported functions on $V.$

  The $\ell^p$ spaces, $\ell^p(V,\nu)$ and $\ell^p(E,\mu),$ for $p\in[1,\infty]$, are defined routinely, and the $\ell^p$ norms of functions are denoted by $\|\cdot\|_{\ell^p(V,\nu)}$ and $\|\cdot\|_{\ell^p(E,\mu)}$ respectively. More specifically, given a discrete measure space $(X,\pi),$ where $\pi$ is a measure defined on the $\sigma$-algebra of all subsets of $X,$ we define $\ell^p$ spaces based on $(X,\pi)$ as: for $p\in[1,\infty),$ the $\ell^p$ norm of a function $f\in \R^X$ is given by $$\|f\|_{\ell^p(X,\pi)}:=\left(\sum_{x\in X}|f(x)|^p\pi(x)\right)^{1/p}\in \R\cup\{\infty\},$$ and the space of $\ell^p$ summable functions
\begin{equation*}
\ell^p(X,\pi):=\left\{f\in\R^X|\ \|f\|_{\ell^p(X,\pi)}<\infty\right\}.
\end{equation*}
For $p=\infty,$ $\|f\|_{\ell^\infty(X,\pi)}=\sup_{x\in X}|f(x)|$ and  $\ell^\infty(X,\pi):=\left\{f\in\R^X|\ \|f\|_{\ell^\infty(X,\pi)}<\infty\right\}.$ This applies to $(X,\pi)=(V,\nu)$ and $(E,\mu)$ for a weighted graph $(V,E,\nu,\mu).$
For convenience, we symmetrically extend the edge weight function $\mu:E\to \R_{+}$ to $$\mu:V\times V\to \R_{\geq 0}, (x,y)\mapsto \mu_{xy}$$ by setting  $\mu_{xy}=0$ if $\{x,y\}\not\in E$.
In this way, we may write $\sum_{y}\mu_{xy}f(y)$ for $\sum_{y\sim x}\mu_{xy}f(y),$ for $x\in V.$

To a weighted graph $G$, one associates  a Dirichlet form
\begin{eqnarray*}
Q: {\mathcal C}_0(V)\to \R, \quad  f\mapsto Q(f):=\frac12\sum_{x,y}\mu_{xy}(f(x)-f(y))^2
\end{eqnarray*}
with respect to the Hilbert space $\ell^2(V,\nu),$ see Keller and Lenz \cite{KellerLenz12}. For locally finite graphs, the associated generator, called Laplacian, acts as
\begin{equation}\label{def:laplacian}
\Delta_{\nu,\mu} f(x)=\frac{1}{\nu_x}\sum_{y}\mu_{xy}(f(x)-f(y))
\end{equation}
for  $f\in {\mathcal C}_0(V).$ The Laplacian $\Delta_{\nu,\mu}$ is a nonnegative self-adjoint operator on $\ell^2(V,\nu)$ that generates  a semigroup $e^{-t\Delta_{\nu,\mu}}$ acting on $\ell^p(V,\nu)$ for $p\in [1,\infty].$
 In this paper, we always impose the following assumption on the weighted graph $G$
\begin{equation}\label{e:bounded Laplacian}\tag{$BLap$}
M:=\sup_{x\in V}\frac{\deg_x}{\nu_x}<\infty,
\end{equation}
where $\deg_x:=\sum_{y\in V}\mu_{xy}.$ This condition is equivalent to the fact that the Laplacian $\Delta_{\nu,\mu}$ is bounded on $\ell^2(V,\nu),$ see \cite[Lemma~1]{Davies93} or \cite{KellerLenz12}.
A specific choice of the measure $\nu,$ $\nu\equiv \deg$ on $V$, is of particular interest; the associated generator $\wt{\Delta}_\mu:=\Delta_{\deg,\mu}$ is then called the \emph{normalized Laplacian.} For simplicity, we omit the subscripts of $\Delta_{\nu,\mu}$ when the measures $\nu,\mu$ are clear from the context.


In order to generalize  results on manifolds to the discrete setting, the first issue is to give an appropriate definition of $\||\nabla f|\|_{\ell^p}$ for $f\in \R^V$ on graphs analogous to $\||\nabla f|\|_{L^p(M,\vol)}$ on manifolds. The common definition adopted by many authors, e.g. \cite{Russ00,Dungey08}, is the following. Given $f\in \R^V,$ define the (length of the) gradient of $f\in \R^V$ as a function on $V:$
\begin{equation}\label{eq:grad}
|\nabla f|\in\R^V,\quad |\nabla f|(x)
:=\left(\frac{1}{2\nu_x}\sum_{y}|f(x)-f(y)|^2\mu_{xy}\right)^{1/2},\quad \forall x\in V.
\end{equation}
This definition coincides with the definition induced by the ``carr\'e du champ" on graphs, see \cite{Lin10,BakryGentilLedoux14,BHLLMY15}. Then the $\ell^p$ norm of the gradient of a function $f$ is defined as
$\||\nabla f|\|_{\ell^p(V,\nu)}$. By mimicking the Riemannian setting, one says that the Riesz transform is bounded on $\ell^p$ for a weighted graph $G=(V,E,\nu,\mu)$ if:
\begin{equation}\label{eq:wrong riesz}\tag{$R_{p,G}$}
\||\nabla f|\|_{
\ell^p(V,\nu)}\leq C\|\Delta^{1/2}f\|_{\ell^p(V,\nu)}, \ \ \ \forall f\in {\mathcal C}_0(V).
\end{equation}
Here is an alternative definition of the $\ell^p$ norm of the gradient of functions. Given a function $f\in \R^V,$ define the difference of $f$ as
$$
D f:V\times V\to\R,\quad (x,y)\mapsto D_{xy} f:=f(y)-f(x).
$$
This gives a well-defined function on edges,  $$|D f|:E\to \R, \quad E\ni\{x,y\}\mapsto |D_{xy}f|,$$ which mimics the (absolute values of the) directional derivatives of a function in various directions on manifolds. Hence the $\ell^p$ norm of the gradient of a function on graphs can also be defined as
\begin{equation}\label{Gradient-edge}
\||D f|\|_{\ell^p(E,\mu)}^p=\sum_{\{x,y\}\in E}\mu_{xy}|f(x)-f(y)|^p.
\end{equation}
Observe that for $p=2$ there holds $\||D f|\|_{\ell^2(E,\mu)}^2=Q(f)$ for any $f\in {\mathcal C}_0(V)$. However, for $p\neq 2$, the significant difference between  $\||D f|\|_{\ell^p(E,\mu)}$ and $\||\nabla f|\|_{\ell^p(V,\nu)}$ is that the former does not depend on the weight $\nu,$ while the latter does. 
We will explore the relation between the two norms in  detail  in Lemmas \ref{l:bounded operator} and \ref{lem:equiv} below. Moreover, we shall see that they are not equivalent even if a local doubling condition is assumed, see Example \ref{ex:non equiv}.

\medskip
\noindent{\bf Notation:} For convenience, the constants $C$ in the context may change from line to line.
Two functionals
$F,H:X\to \R_{\geq0}$ on a space of functions $X$  are called equivalent, denoted by $$F(f)\approx H(f),\quad \forall f\in X,$$ if there are uniform constants $c$ and $C$ such that $$cH(f)\leq F(f)\leq CH(f),\quad \forall f\in X.$$

\subsection{Main results}
Apart from preliminary Section \ref{s:setting}, the present paper decomposes into three parts, in which we study the discrete counterparts of known results for the Riesz transform and related topics on Riemannian manifolds.
Section~\ref{s:Modified Stein} is devoted to the proof of  an interpolation or multiplicative inequality by a modification of an argument by Stein, see the proof of Theorem~\ref{thm:main thm1}.
Section \ref{LPS} is about estimates for Littlewood-Paley-Stein functions on graphs.
Comparing with previous results, we adopt the \lq\lq gradient\rq\rq ~$|Df|$ instead of $|\nabla f|$. In fact, we justify that $|Df|$ is a more natural definition in the general setting by providing two counterexamples (see Examples \ref{ex:non equiv}, \ref{ex:non MIp}). In Section \ref{s:Riesz transform}, we study the $L^p$ boundedness of Riesz transforms on graphs with positive spectral gap of the Laplacian.

\medskip

%
Let us  first recall the analogue on a Riemannian manifold $M$ of some properties that we are going to study on graphs.
The first one is an $L^p$ interpolation or multiplicative inequality:
\begin{equation}\label{prop:MIp}
\||\nabla f|\|_{L^p(M,\vol)}^2\leq C_p\|f\|_{L^p(M,\vol)}\|\Delta_M f\|_{L^p(M,\vol)},\quad \forall f\in {\mathcal C}_0^{\infty}(M).
\end{equation}
The second one is an $L^p$ estimate of the gradient of the heat semigroup:
\begin{equation}\label{prop:GpM}
\||\nabla e^{-t\Delta_M}f|\|_{L^p(M,\vol)}^2
\leq C_pt^{-1/2}\|f\|_{L^p(M,\vol)},\quad \forall f\in {\mathcal C}_0^{\infty}(M), t>0.
\end{equation}
It is known that \eqref{prop:MIp} is equivalent to \eqref{prop:GpM} for any $p\in (1,\infty)$, see
 \cite{CoulhonSikora10}. Both of them are weaker than the boundedness of Riesz transform, i.e. \eqref{eq:riesz manifold} implies \eqref{prop:MIp} and \eqref{prop:GpM}. It has been proved in the Riemannian case that \eqref{prop:MIp} and \eqref{prop:GpM} always hold for $p\in(1,2]$ (see \cite{Chen15}). However, it is unknown for \eqref{eq:riesz manifold} in the general case, i.e. it is still an open problem whether there is a complete Riemannian manifold on which the Riesz transform is unbounded for some $p\in (1,2)$ (see \cite[Conjecture 1.1]{CoulhonDuong03}).  For positive results on \eqref{eq:riesz manifold} in this range, we refer to \cite{CoulhonDuong99,ChenCoulhonFeneuilRuss15} where  some pointwise upper bound of the heat kernel together with a global volume doubling property is  needed.

Using  definition \eqref{eq:grad}, Dungey \cite{Dungey08} could extend these results on manifolds to graphs under some additional assumption.
A weighted graph $(V,E,\deg,\mu)$ with the normalized Laplacian $\wt{\Delta}_\mu$ is said to satisfy the \emph{local doubling condition} if for some constant $c_0>1,$
\begin{equation}\label{a:local doubling}\tag{$LD$}
\sum_{y\sim x}\deg_y\leq c_0\deg_x,\quad\forall x\in V.
\end{equation}

This condition is equivalent to the conjunction of the following two properties, see \cite[p.571]{CoulhonGrigoryanZucca05}:
\begin{enumerate}[(a)]
\item There exists $c_1\geq 1$ such that $\deg_y\leq c_1\deg_x$ for all $x\sim y\in V.$
\item\label{a:uniformly local finite} The graph is uniformly locally finite, i.e. $\sup_{x\in V}\sharp\{y\in V:y\sim x\}<\infty.$
\end{enumerate}
Property $(b)$ means that the combinatorial degree is uniformly bounded from above, which is a restrictive assumption for graphs. Under the assumption  \eqref{a:local doubling}, Dungey proved that for any $p\in (1,2],$ $f\in {\mathcal C}_0(V), t>0,$
\begin{equation}\label{eq:Dungey MIp}
\||\nabla f|\|^2_{\ell^p(V,\deg)}\leq C_p \|f\|_{\ell^p(V,\deg)}\|\wt{\Delta}_\mu f\|_{\ell^p(V,\deg)},
\end{equation}
\begin{equation}\label{eq:Dungey Gp}
\||\nabla e^{-t\wt{\Delta}_\mu}f|\|_{\ell^p(V,\deg)}\leq C'_p t^{-1/2}\|f\|_{\ell^p(V,\deg)},
\end{equation}
where $C_p,C'_p$ are constants depending only on $p,$ see \cite[Theorem~1.1~and~Corollary~1.2]{Dungey08}.


Our first observation is that in order to obtain Dungey's results \eqref{eq:Dungey Gp} and \eqref{eq:Dungey MIp} on graphs, the assumption of local doubling condition \eqref{a:local doubling} is somehow necessary. In fact, without this constraint we may construct a counterexample to \eqref{eq:Dungey MIp} for $p\in(1,2),$ see Example \ref{ex:non MIp}. This also provides a counterexample to the boundedness of Riesz transforms for $p\in (1,2)$ on graphs (see Remark \ref{counter} below). However, we don't know any results in this direction on manifolds as mentioned before.

Recall  that \eqref{prop:MIp} and \eqref{prop:GpM} hold for $1<p\le 2$ on any complete manifold without any local assumption (see \cite[Proposition 2.7]{Chen15}). Hence in the setting of graphs,  \eqref{a:local doubling} is not a natural assumption and in fact quite restrictive. In order to circumvent this restriction, we will use  the definition $\||D f|\|_{\ell^p(E,\mu)}$ in place of $\||\nabla f|\|_{\ell^p(V,\deg)}$. Now we state the following theorem, for which the proof will be given in Section~\ref{s:Modified Stein}.

\begin{thm}\label{thm:main thm1}
Let $G=(V,E,\nu,\mu)$ be a weighted graph satisfying \eqref{e:bounded Laplacian} and $p\in (1,2].$ Then for any $f\in {\mathcal C}_0(V), t>0,$
\begin{equation}\label{eq:mainthm 2}\tag{$MI_{p,G}$}
\||D f|\|^2_{\ell^p(E,\mu)}\leq C_p \|f\|_{\ell^p(V,\nu)}\|\Delta f\|_{\ell^p(V,\nu)},
\end{equation}
 \begin{equation}\label{eq:mainthm 1}\tag{$G_{p,G}$}
 \||D e^{-t\Delta}f|\|_{\ell^p(E,\mu)}\leq C_p t^{-1/2}\|f\|_{\ell^p(V,\nu)}.
 \end{equation}
 \end{thm}
In fact,  properties \eqref{eq:mainthm 1} and  \eqref{eq:mainthm 2} are equivalent for any weighted graph and any $p\in(1,\infty),$ see Lemma~\ref{GpMIp}.



\medskip

Next we study the $L^p$ boundedness of so-called Littlewood-Paley-Stein functions on graphs. In continuous settings, they have received much attention in harmonic analysis and also probability theory, see e.g. \cite{Stein70,Bakry87,CoulhonDuongLi03, BBL16} and references therein.
The ``vertical" Littlewood-Paley-Stein operator $\mathcal{H}_{\Delta_M}$ on a Riemannian manifold $M$ is defined for functions $f\in {\mathcal C}_0^{\infty}(M)$ by
$$
\mathcal{H}_{\Delta_M}f(x):=\left(\int_0^\infty |(\nabla e^{-t\Delta_M}f)(x)|^2 dt\right)^\frac{1}{2},\quad x\in M.
$$
Stein's argument \cite[Chapter~II]{Stein70} can be used to derive the Littlewood-Paley-Stein estimate on Riemannian manifolds, see \cite[Theorem~1.2]{CoulhonDuongLi03}: $\mathcal{H}_{\Delta_M}$ extends to a bounded (sublinear) operator in $L^p(M,\vol),$ for $p\in(1,2],$ i.e.
$$
\|\mathcal{H}_{\Delta_M}f\|_{L^p}\leq C_p\|f\|_{L^p},\quad \forall f\in {\mathcal C}_0^{\infty}(M).
$$

An analogous result for normalized Laplacians on graphs was obtained by Dungey \cite[Theorem~1.3]{Dungey08}. That is,
let $G$ be a weighted graph with normalized Laplacian $\wt{\Delta}_\mu.$ For any $0\leq f\in {\mathcal C}_0(V),$
%
define
$$
(\mathcal{H}_{p,\wt{\Delta}_\mu}f)(x):=\left(\int_0^\infty \Gamma_p(e^{-t\wt{\Delta}_\mu}f)(x)dt\right)^\frac{1}{2},\quad x\in V,
$$
 where $\Gamma_p(u):=pu\wt{\Delta}_\mu u-u^{2-p}\wt{\Delta}_\mu (u^p)$ is a modified gradient.

Then for $p\in (1,2],$ there exists a constant $C_p$ such that
$$
\|\mathcal{H}_{p,\wt{\Delta}_\mu}f\|_{\ell^p(V,\deg)}\leq C_p\|f\|_{\ell^p(V,\deg)},\quad \forall 0\leq f\in {\mathcal C}_0(V).
$$

The definition of $\mathcal{H}_{p,\wt{\Delta}_\mu}$ involves the technical term $\Gamma_p(f)$ which differs for various $p\in (1,2].$ We introduce a modified Littlewood-Paley-Stein function on edges which does not involve $p$ by replacing $\Gamma_p$ with $D$. Define $\mathcal{H}f\in \R^E$ for any $ f\in {\mathcal C}_0(V)$ as
$$
(\mathcal{H}f)(\{x,y\}):=\left(\int_0^\infty  |D(e^{-t\Delta}f)|^2(\{x,y\})dt\right)^\frac{1}{2},\quad \{x,y\}\in E.
$$
In addition, given $a\in \R,$ define $\mathcal{H}_af\in \R^E$ for any $f\in {\mathcal C}_0(V)$ as
$$
(\mathcal{H}_af)(\{x,y\}):=\left(\int_0^\infty e^{at} |D(e^{-t\Delta}f)|^2(\{x,y\})dt\right)^\frac{1}{2},\quad \{x,y\}\in E.
$$
Here $a\in\R$ is a parameter that one uses to improve estimates for Littlewood-Paley-Stein functions on graphs possessing a positive spectral gap, following \cite{Lohoue87} and \cite[Theorem~1.3]{CoulhonDuongLi03}. We say a weighted graph $G=(V,E,\nu,\mu)$ has a positive spectral gap if the infimum of the spectrum of the Laplacian is positive, i.e. $\inf\sigma(\Delta)>0,$ where $\sigma(\Delta)$ denotes the $\ell^2(V,\nu)$ spectrum of the Laplacian $\Delta$. Using the same trick as in the proof of Theorem \ref{thm:main thm1}, we can show the following.
\begin{thm}\label{thm:main thm2}
Let $G=(V,E,\mu,\nu)$ be a weighted graph satisfying \eqref{e:bounded Laplacian} and let $p\in (1,2].$ Then there exists a constant $C$ such that for any $t>0,$
\begin{equation}\label{pleq}
\|\mathcal{H}f\|_{\ell^p(E,\mu)}\leq C\|f\|_{\ell^p(V,\nu)},\quad \forall f\in {\mathcal C}_0(V).
\end{equation}
Moreover, if $G$ has a positive spectral gap $\lambda_1=\inf\sigma(\Delta)>0$, then for any $a<2\lambda_1(p-1),$
\begin{equation}\label{eq:alpha pleq}
\|\mathcal{H}_a(f)\|_{\ell^p(E,\mu)}\leq C\|f\|_{\ell^p(V,\nu)},\quad \forall f\in {\mathcal C}_0(V).
\end{equation}
\end{thm}

\medskip

Finally, we consider the $L^p$ boundedness of Riesz transforms on graphs with positive spectral gap of the Laplacian.
On Riemannian manifolds, the positive spectral gap of the Laplacian yields a general conclusion that the Riesz transform is bounded for any $1<p<\infty$ under some assumptions of volume growth properties and short time heat kernel estimates, see \cite{AuscherCoulhonDuongHofmann04} or Section~\ref{s:Riesz transform}. Furthermore, under the same assumptions, Ji, Kunstmann and Weber \cite{JiKunstmannWeber10} generalized these results to Riemannian manifolds possessing a positive essential spectral gap.

We will generalize these results to graphs endowed with bounded Laplacians. 



\begin{thm}\label{thm:main theorem}
Let $G=(V,E,\nu,\mu)$ be a weighted graph satisfying \eqref{e:bounded Laplacian}. Suppose that $G$ has a positive spectral gap, then for any $p\in (1,\infty),$
$$
\||Df|\|_{\ell^p(E,\mu)}\approx \|\Delta^{1/2} f\|_{\ell^p(V,\nu)}\approx \|f\|_{\ell^p(V,\nu)},\ \ \forall f\in {\mathcal C}_0(V).
$$
Moreover, for $p\in [2,\infty),$
$$
\||\nabla f|\|_{\ell^p(V,\nu)}\leq C\|\Delta^{1/2}f\|_{\ell^p(V,\nu)},\ \ \forall f\in {\mathcal C}_0(V).
$$
\end{thm}
It is remarkable that on graphs it happens that
$$
\||D f|\|_{\ell^p(E,\mu)}\approx \|f\|_{\ell^p(V,\nu)},\ \forall f\in {\mathcal C}_0(V),
$$
which has no counterpart in the continuous setting: by a local scaling argument it cannot be true on manifolds since the order of $|Df|$ is strictly higher than that of $f$.

The next theorem generalizes the result of \cite{JiKunstmannWeber10}. Given a weighted graph $G,$ let $\sigma_{\es}(\Delta)$ denote the essential spectrum of $\Delta.$ Recall that $\lambda\in \sigma_{\es}(\Delta)$ if and only if $\lambda$ is either an accumulation point of the set $\sigma(\Delta)$ or an eigenvalue of infinite multiplicity, see for instance \cite{ReedSimon80}. We say that $G$ has a positive essential spectral gap if $\inf\sigma_{\es}(\Delta)>0$.

\begin{thm}\label{thm:generalization}
Let $G=(V,E,\nu,\mu)$ be a graph satisfying \eqref{e:bounded Laplacian}. Suppose that $G$ has a positive essential spectral gap and $0\in \sigma(\Delta),$ then $\nu(V):=\sum_x\nu_x<\infty$ and for any $p\in (1,\infty)$
$$
\||D f|\|_p\approx \|f-\overline{f}\|_p\approx \|\Delta^{1/2} f\|_p,\ \ \forall f\in {\mathcal C}_0(V),
$$
where $\overline{f}=\frac{1}{\nu(V)}\sum_{x\in V}f(x)\nu_x.$
\end{thm}
%


\section{Preliminaries}\label{s:setting}
Under the assumption  \eqref{e:bounded Laplacian}, one can show that the Laplacian $\Delta$, as well as the gradients $\nabla$ and $D$, are bounded operators on $\ell^p(V,\nu)$ for any $p\in [1,\infty].$
\begin{lemma}\label{l:bounded operator} Let $G=(V,E,\nu,\mu)$ be a weighted graph satisfying \eqref{e:bounded Laplacian}. Then for  $p\in[1,\infty],$
\begin{eqnarray}&&\||D f|\|_{\ell^p(E,\mu)}\leq \|f\|_{\ell^p(V,\nu)},\label{eq:bound1}\\
&&\|\Delta f\|_{\ell^p(V,\nu)}\leq \|f\|_{\ell^p(V,\nu)},\label{eq:bound2}
\end{eqnarray}
for all $f\in {\mathcal C}_0(V)$. Moreover,
 \begin{eqnarray}
&&\||Df|\|_{\ell^p(E,\mu)} \le C\||\nabla f|\|_{\ell^p(V,\nu)}, \quad p\in[1,2] \label{eq:bound3}\\
&&\||\nabla f|\|_{\ell^p(V,\nu)}\leq C\||Df|\|_{\ell^p(E,\mu)}, \quad p\in[2,\infty]\label{eq:bound4}
\end{eqnarray}
for all $f\in {\mathcal C}_0(V)$.
\end{lemma}

\begin{proof} For the first assertion,
 \begin{eqnarray*} \||Df|\|_{\ell^p(E,\mu)}^p&=&\sum_{x,y}\mu_{xy}|f(x)-f(y)|^p\leq C_p\sum_{x,y}\mu_{xy}(|f(x)|^p+|f(y)|^p)\\
  &=&2C_p\sum_{x,y}|f(x)|^p\mu_{xy}=2C_p\sum_x|f(x)|^p\deg_x\\&\leq&2C_pM\|f\|_{\ell^p(V,\nu)}^p,\end{eqnarray*} where we have used \eqref{e:bounded Laplacian} in the last inequality.

  By H\"older's inequality and \eqref{e:bounded Laplacian},
  \begin{eqnarray*}
    \|\Delta f\|_{\ell^p(V,\nu)}^p
    &=&\sum_{x}\nu_x\left(\frac{\deg_x}{\nu_x}\right)^p\left|\sum_{y}\frac{\mu_{xy}}{\deg_x}(f(x)-f(y))\right|^p\\
    &\leq &\sum_x\nu_x\left(\frac{\deg_x}{\nu_x}\right)^p\sum_{y}\frac{\mu_{xy}}{\deg_x}|f(x)-f(y)|^p\\
    &\leq &M^{p-1}\sum_{x,y}\mu_{xy}|f(x)-f(y)|^p=2M^{p-1}\||D f|\|_{\ell^p(E,\mu)}^p.
  \end{eqnarray*}

 For the second assertion, we only prove the case for $p\in[1,2]$ (a similar argument applies for $p\in [2,\infty]$). By H\"older's inequality and \eqref{e:bounded Laplacian},
  \begin{eqnarray*}
    \||\nabla f|\|_{\ell^p(V,\nu)}^p
    &=&\sum_{x}\nu_x\left(\frac{\deg_x}{\nu_x}\right)^{\frac{p}{2}}\left(\sum_{y}\frac{\mu_{xy}}{\deg_x}(f(x)-f(y))^2\right)^{\frac{p}{2}}\\
    &\geq&\sum_{x}\nu_x\left(\frac{\deg_x}{\nu_x}\right)^{\frac{p}{2}}\sum_{y}\frac{\mu_{xy}}{\deg_x}|f(x)-f(y)|^p\\
    &\geq&M^{\frac{p}{2}-1}\sum_{x,y}\mu_{xy}|f(x)-f(y)|^p.
  \end{eqnarray*}

\end{proof}


Here is another property of bounded Laplacians.

\begin{lemma}\label{lem:laplace zero}
  Let $G=(V,E,\nu,\mu)$ be a weighted graph satisfying \eqref{e:bounded Laplacian} and $f\in \ell^1(V,\nu).$ Then
  \begin{equation}\label{eq:laplace zero}
    \sum_{x\in V}\Delta f(x)\nu_x=0.
  \end{equation}
\end{lemma}
\begin{proof}
  By equation \eqref{eq:bound2}, $\Delta f\in \ell^1(V,\nu).$ Then \begin{equation*}
    \sum_{x\in V}\Delta f(x)\nu_x=\sum_{x\in V}\nu_x\left(\sum_{y}\frac{\mu_{xy}}{\nu_x}(f(x)-f(y))\right)<\infty.
  \end{equation*} Consider
  \begin{eqnarray*}
    &&\sum_{x\in V}\nu_x\sum_{y}\frac{\mu_{xy}}{\nu_x}|f(x)-f(y)|=\sum_{x,y}\mu_{xy}|f(x)-f(y)|\\
    &\leq &\sum_{x,y}\mu_{xy}(|f(x)|+|f(y)|)\leq C\sum_{x\in V}|f(x)|\nu_x<\infty.
  \end{eqnarray*}
  By Fubini's theorem, \begin{eqnarray*}\sum_{x\in V}\Delta f(x)\nu_x&=&\sum_{x,y}\mu_{xy}(f(x)-f(y))\\
  &=&\sum_{x,y}\mu_{yx}(f(y)-f(x))=-\sum_{x,y}\mu_{xy}(f(x)-f(y)),\end{eqnarray*} where we have swaped $x$ and $y$ in the second equality.
  This proves the lemma.
\end{proof}

Going back to the case of normalized Laplacians, the following condition is stronger than \eqref{a:local doubling}:
\begin{equation}\label{a:local uniform}
  \inf_{y\sim x}\frac{\mu_{xy}}{\deg_x}>0.
\end{equation}
Indeed, let $\delta=\inf_{y\sim x}\frac{\mu_{xy}}{\deg_x},$ for any $x\sim y,$
$$\delta \deg_x\leq \mu_{xy}\leq \deg_x.$$ So the number of neighbors of $x$ is bounded above by $1/{\delta}.$ Note that for any $y\sim x,$ $\deg_y\leq \frac{1}{\delta}\mu_{xy}\leq \frac{1}{\delta}\deg_x.$ This implies $(LD).$

\begin{lemma}\label{lem:equiv}
  Let $(V,E,\deg,\mu)$ be a weighted graph satisfying \eqref{a:local uniform}. Then for any $p\in [1,\infty)$
  $$\||\nabla f|\|_{\ell^p(V,\deg)}\approx\||Df|\|_{\ell^p(E,\mu)},$$
for all  $f\in {\mathcal C}_0(V)$.
\end{lemma}
\begin{proof}
  Note that \eqref{a:local uniform} implies that the graph is uniformly locally finite and $$0<\delta\leq \frac{\mu_{xy}}{\deg_x}\leq 1,\quad \forall x\sim y.$$ Then for any $p\in [1,\infty)$
  \begin{eqnarray*}\||\nabla f|\|_{\ell^p(V,\deg)}^p&=&\sum_{x}\deg_x\left(\sum_{y}\frac{\mu_{xy}}{\deg_x}(f(x)-f(y))^2\right)^{\frac{p}{2}}\\
  &\approx&\sum_{x}\deg_x\left(\sum_{y}\frac{\mu_{xy}}{\deg_x}|f(x)-f(y)|^p\right)=\||Df|\|_{\ell^p(E,\mu)}^p.
  \end{eqnarray*}
\end{proof}

The following example shows that the assumption \eqref{a:local uniform} is necessary for the previous lemma to hold. More precisely, Lemma \ref{lem:equiv} does not hold if one replaces \eqref{a:local uniform} with \eqref{a:local doubling} for $p\in(1,2).$
\begin{example}\label{ex:non equiv}
  Let $\epsilon>0.$ Let $(\Z,E,\deg,\mu)$ be an infinite one-dimensional lattice where $\{i,i+1\}\in E,$ for any $i\in\Z,$  $\mu_{0,1}=\epsilon$ and $\mu_{i,i+1}=1$ if $i\neq 0.$ This graph satisfies \eqref{a:local doubling} with a doubling constant independent of $\epsilon.$ However, one can show that the following statement cannot be true: for $p\in(1,2),$
$$
\||\nabla f|\|_{\ell^p(V,\deg)}\leq C\||Df|\|_{\ell^p(E,\mu)},\quad \forall f\in {\mathcal C}_0(V),
$$
where $C$ is a uniform constant independent of $\epsilon.$ We argue by contradiction. For any $k\in\N,$ plugging into the above estimate \begin{equation*}
f_k(i)=\left\{\begin{array}{ll}
0, & i\leq 0, \\
\max\{1-\frac{i}{k},0\}, & i\geq1,
\end{array}\right.
\end{equation*}
and letting $k\to\infty,$ we have $\epsilon^{\frac12}\leq C\epsilon^{\frac{1}{p}}$ which is impossible for $p\in(1,2)$ and small $\epsilon.$
\end{example}

\bigskip

\section{Modified Stein's argument on graphs}\label{s:Modified Stein}
For a suitable function $f$ defined on a Riemannian manifold, it holds for $1<p<\infty$ that
$$
\Delta_M(f^p)=pf^{p-1}\Delta_M f-p(p-1)f^{p-2}|\nabla f|^2.
$$
Hence
$$
p(p-1)|\nabla f|^2=pf\Delta_M f-f^{2-p}\Delta_M (f^p).
$$
This is a chain rule in the continuous setting which  usually fails in the discrete setting. To overcome this difficulty on graphs, the so-called \lq\lq pseudo-gradient\rq\rq ~ was introduced in \cite{Dungey08}. That is,  for a suitable function $f:V\to\R_{+}$, set
$$
\Gamma_p(f):=pf\Delta f-f^{2-p}\Delta (f^p),
$$
where $\Delta$ is the Laplacian on functions defined on vertices. For brevity, we introduce a two-variable function
$$
\gamma_p(\alpha,\beta):=p\alpha(\alpha-\beta)-\alpha^{2-p}(\alpha^p-\beta^p),\forall \alpha,\beta\geq0.
$$
Then for any $x\in V,$
\begin{eqnarray}
 \quad\quad \Gamma_p(f)(x)&=&\sum_{y}\frac{\mu_{xy}}{\nu_x}[pf(x)(f(x)-f(y))-f(x)^{2-p}(f(x)^p-f(y)^p)]
  \label{eq:formula Gamma}\\
  &=&\sum_{y}\frac{\mu_{xy}}{\nu_x}\gamma_p(f(x),f(y)).\nonumber
\end{eqnarray}
Note that the function $\gamma_p(\alpha,\beta)$ is not symmetric in variables $\alpha$ and $\beta.$ The following calculus lemma in \cite[Lemma~3.2]{Dungey08} is useful.
\begin{lemma}
For $1<p\leq 2,$ $\alpha,\beta\geq 0,$ \begin{equation*}
\gamma_p(\alpha,\beta)=\left\{\begin{array}{ll}
p(p-1)(\alpha-\beta)^2\int_{0}^1\frac{(1-u)\alpha^{2-p}}{((1-u)\alpha+u\beta)^{2-p}}du, & (\alpha,\beta)\neq (0,0), \\
0, & (\alpha,\beta)= (0,0).
\end{array}\right.
\end{equation*}
\end{lemma}

The pseudo-gradient $\Gamma_p(f)$ plays the role of a substitute to the gradient squared of $f$ in the discrete setting, while its expression is rather complicated and depends on $p\in(1,2].$ The key property of $\Gamma_p(f)$ is the following (see \cite[pp 122]{Dungey08}): for any $f\geq 0$ and $x\in V,$
$$
0\leq \Gamma_p(f)(x)\leq (p-1)|\nabla f|^2(x).
$$

 As pointed out in \cite{Dungey08}, the downside of the pseudo-gradient is that it cannot be pointwise bounded from below by the gradient squared. Indeed, one cannot hope that there is a uniform constant $C$  such that
$$
\Gamma_p(f)(x)\geq C |\nabla f|^2(x),\quad\quad \forall f\geq 0, x\in V.
$$
This can be seen as follows (see \cite[pp 117]{Dungey08}): let $f\geq 0$ such that $f(x)=0$ and $f(y)>0$ for some $y\sim x,$ then $\Gamma_p(f)(x)=0$ and $|\nabla f|^2(x)>0.$ To overcome this defect, Dungey introduced the local doubling condition \eqref{a:local doubling} for graphs and used a local average argument to show that $|\nabla f|^2$ is dominated by a local average of $\Gamma_p(f)$ (see \cite[Proposition~3.1~and~Lemma~3.3]{Dungey08}).


We first provide a counterexample to show that
\eqref{a:local doubling} is necessary in some sense if we want to use $$\||\nabla f|\|_{\ell^p(V,\deg)}$$ as the $\ell^p$ energy of the function $f.$
More precisely, if we do not assume \eqref{a:local doubling}, we can construct a graph which does not satisfy  property \eqref{eq:Dungey MIp} for $p\in(1,2).$

Let us record a necessary condition for    \eqref{eq:Dungey MIp}.
\begin{prop} Let $p\in [1,\infty)$.
Let $(V,E,\deg,\mu)$ be a normalized weighted graph. If  \eqref{eq:Dungey MIp} holds, then for any $x\in V,$
\begin{equation}\label{eq:char MIp}
\left(\deg_x+\sum_{y\sim x}\mu_{xy}^{\frac{p}{2}}\deg_y^{1-\frac{p}{2}}\right)^2
\leq C_p\deg_x\br{\deg_x+\sum_{y\sim x}\mu_{xy}^p\deg_y^{1-p}}.
\end{equation}
\end{prop}
\begin{proof}
This follows from plugging into \eqref{eq:Dungey MIp} $f(y)=\delta_x(y).$
\end{proof}

\begin{example}\label{ex:non MIp}
Let $T=(V,E)$ be an infinite tree with the expanding sequence $\{d_i\}_{i=0}^\infty,$ i.e.  each vertex in the layer $i$ has $d_i$ successors where $d_i=1$ for $i=0,1,2,3,$ $d_{2n}=1$ and $d_{2n+1}=2n+1$ for $n\geq 2.$ Set $\mu\equiv1$ on $E$ and $\nu\equiv\deg$ on $V.$ This graph does not satisfy \eqref{eq:char MIp} for $p\in(1,2)$, hence gives us a counterexample to \eqref{eq:Dungey MIp}.
\end{example}
Indeed, take $x\in T$ in the layer $2n$, $n\ge 3$. Then $\deg_x=2$ and there are two neighbors of $x$ in the layers $2n-1$ and $2n+1$, which we denote by $y_{2n-1}$ and $y_{2n+1}$ respectively. Note that $\deg_{y_{2n-1}}= 2n$ and $\deg_{y_{2n+1}}= 2n+2$, hence for $p\in(1,2),$
$$
\left(\deg_x+\sum_{y\sim x}\mu_{xy}^{\frac{p}{2}}\deg_y^{1-\frac{p}{2}}\right)^2
=\left(2+(2n)^{1-\frac{p}{2}}+(2n+2)^{1-\frac{p}{2}}\right)^2 \geq C n^{2-p},
$$
and
$$
\deg_x\br{\deg_x+\sum_{y\sim x}\mu_{xy}^p\deg_y^{1-p}}
=2 \br{2+(2n)^{1-p}+(2n+2)^{1-p}} \leq 8.
$$
The above two estimates contradict \eqref{eq:char MIp} as $n$ goes to infinity.

\begin{rem} \label{counter}
Recall that on $\ell^p(V, \nu)$ with $\nu\equiv\deg$, the Riesz transform is bounded if
$$
\||\nabla f|\|_{\ell^p(V,\deg)}\leq C_p \|\wt{\Delta}_\mu^{1/2} f\|_{\ell^p(V,\deg)}.
$$
This implies \eqref{eq:Dungey MIp} by an interpolation argument following verbatim \cite{Komatsu66, Carron07}. Hence Example \ref{ex:non MIp} is also a counterexample to the boundedness of the Riesz transform for all $p\in (1,2)$.
\end{rem}

The following lemma shows that the symmetrized version of $\gamma_p(\alpha,\beta)$ is equivalent, up to some constant, to $(\alpha-\beta)^2$ for any $\alpha,\beta\geq 0.$
\begin{lemma}\label{lem:sym} For any $\alpha,\beta\geq 0,$ $p\in(1,2],$
   \begin{equation}
   \label{eq:symm gamma}(p-1)(\alpha-\beta)^2\leq \gamma_p(\alpha,\beta)+\gamma_p(\beta,\alpha)\leq p(\alpha-\beta)^2.\end{equation}
\end{lemma}
\begin{proof}
  By an elementary computation,
  \begin{equation*}
    \gamma_p(\alpha,\beta)+\gamma_p(\beta,\alpha)=p(\alpha-\beta)^2-(\alpha^{2-p}-\beta^{2-p})(\alpha^p-\beta^p).
  \end{equation*} It suffices to prove the first inequality (the second one is trivial). The lemma follows from
  \begin{eqnarray*}(\alpha^{2-p}-\beta^{2-p})(\alpha^p-\beta^p)&=&\alpha^2+\beta^2-\alpha^{2-p}\beta^p-\beta^{2-p}\alpha^p\\
  &=&\alpha^2+\beta^2-\alpha\beta\left(\Big(\frac{\beta}{\alpha}\Big)^{p-1}+\Big(\frac{\alpha}{\beta}\Big)^{p-1}\right)\\&\leq&(\alpha-\beta)^2.
  \end{eqnarray*}


\end{proof}


\begin{lemma}\label{lem:key Gamma}
Let $G=(V,E,\nu,\mu)$ be a graph satisfying \eqref{e:bounded Laplacian}. There is a constant $C(p,M)$ such that for any $0\leq f\in {\mathcal C}_0(V)$,
\begin{equation}\label{eq:key estimate}
\||D f|\|_{\ell^p(E,\nu)} \le C\norm{\sqrt{\Gamma_p(f)}}_{\ell^p(V,\nu)}.
\end{equation}
\end{lemma}
\begin{proof}
By the definition of $\Gamma_p,$ the H\"older inequality, and \eqref{e:bounded Laplacian}, we have
\begin{eqnarray}
\norm{\sqrt{\Gamma_p(f)}}_{\ell^p(V,\nu)}^p&=&\sum_{x}\nu_x\left(\frac{\deg_x}{\nu_x}\right)^{\frac{p}{2}}\left(\sum_{y}\frac{\mu_{xy}}{\deg_x}\gamma_p(f(x),f(y))\right)^{\frac{p}{2}}\nonumber
\\ &\geq&
\sum_{x}\nu_x\left(\frac{\deg_x}{\nu_x}\right)^{\frac{p}{2}}\left(\sum_{y}\frac{\mu_{xy}}{\deg_x}\gamma_p^{\frac{p}{2}}(f(x),f(y))\right)\nonumber
\\ &\geq&
M^{\frac{p}{2}-1}\sum_{x,y}\mu_{xy}\gamma_p^{\frac{p}{2}}(f(x),f(y))\nonumber\\
&=&
\frac12 M^{\frac{p}{2}-1}\sum_{x,y}\mu_{xy}\left(\gamma_p^{\frac{p}{2}}(f(x),f(y))+\gamma_p^{\frac{p}{2}}(f(y),f(x))\right)\label{eq:sym1}
\\   &\geq&
c_pM^{\frac{p}{2}-1}\sum_{x,y}\mu_{xy}\left(\gamma_p(f(x),f(y))+\gamma_p(f(y),f(x))\right)^{\frac{p}{2}}\nonumber
\\  &\geq&
c'_pM^{\frac{p}{2}-1}\sum_{x,y}\mu_{xy}|f(x)-f(y)|^p=C\||D f|\|_{\ell^p(E,\mu)}^p,\nonumber
 \end{eqnarray}
 where we have used the symmetrization in \eqref{eq:sym1} and Lemma \ref{lem:sym} in the last inequality. This proves the lemma.
\end{proof}

The following lemma is the equivalence of \eqref{eq:mainthm 1} and \eqref{eq:mainthm 2}, whose proof in the continuous case can be found in \cite[Proposition 2.5]{Chen15}.
\begin{lemma}\label{GpMIp}
  Let $(V,E,\mu,\nu)$ be a weighted graph. Then for any $p\in(1,\infty),$ \eqref{eq:mainthm 1} is equivalent to \eqref{eq:mainthm 2}.
 \end{lemma}

\begin{proof}
First assume \eqref{eq:mainthm 2}. Substituting $f$ by $e^{-t\Delta}f$ in \eqref{eq:mainthm 2} yields
\[
\norm{\abs{D e^{-t\Delta}f}}_{\ell^p(E,\mu)}^2 \leq C \norm{\Delta e^{-t\Delta}f}_{\ell^p(V,\nu)} \norm{e^{-t\Delta}f}_{\ell^p(V,\nu)}.
\]
By analyticity of Markov semigroups, see \cite{Stein70}, \cite[Theorem~1.4.2]{Davies89} and \cite[Theorem~8.4.6]{Davies07},
  we obtain
\[
\norm{\abs{D e^{-t\Delta}f}}_{\ell^p(E,\mu)} \leq C t^{-1/2}\Vert f\Vert_{\ell^p(V,\nu)}.
\]

Conversely assume \eqref{eq:mainthm 1}. For any $f\in {\mathcal C}_0(V)$, write the identity
\[
f=e^{-t\Delta}f+\int_{0}^{t}\Delta e^{-s\Delta}f ds,\,\forall t>0.
\]
Then \eqref{eq:mainthm 1} yields
\begin{eqnarray*}
\norm{\abs{D f}}_{\ell^p(E,\mu)} &\leq&
C\norm{\abs{D e^{-t\Delta}f}}_{\ell^p(E,\mu)}+\left\Vert\int_{0}^{t} \abs{D \Delta e^{-s\Delta}f} ds\right\Vert_{\ell^p(E,\mu)}
\\ &\leq&
C t^{-1/2}\Vert f\Vert_{\ell^p(V,\nu)}+\int_{0}^{t} \norm{\abs{D e^{-s\Delta}\Delta f}}_{\ell^p(E,\mu)} ds
\\ &\leq&
C t^{-1/2}\Vert f\Vert_{\ell^p(V,\nu)}+C t^{1/2} \Vert \Delta f\Vert_{\ell^p(V,\nu)}.
\end{eqnarray*}
Taking $t=\Vert f\Vert_{\ell^p(V,\nu)} \Vert \Delta f\Vert_{\ell^p(V,\nu)}^{-1}$, we get \eqref{eq:mainthm 2}.
\end{proof}

Given a function $0\leq f\in \ell^1(V,\nu)\cap \ell^\infty(V,\nu),$ set $u(t,x)=e^{-t\Delta} f(x).$
Since $\partial_t u=-\Delta u,$ by the definition of $\Gamma_p,$
\begin{equation}\label{eq:Gamma equ}\Gamma_p(u)=-u^{2-p}(\partial_t+\Delta)(u^p).\end{equation}
Combining this lemma with Stein's argument, we can prove Theorem~\ref{thm:main thm1}.

\begin{proof}[\bf Proof of Theorem~\ref{thm:main thm1}]
Since \eqref{eq:mainthm 1} is equivalent to \eqref{eq:mainthm 2}, see Lemma \ref{GpMIp}, we will prove \eqref{eq:mainthm 2}. By Lemma \ref{lem:key Gamma}, it suffices to show that \begin{equation}\label{eq:equ2}
\norm{\sqrt{\Gamma_p(f)}}_{\ell^p(V,\nu)}^2 \leq C \|f\|_{\ell^p(V,\nu)}\|\Delta f\|_{\ell^p(V,\nu)},\quad \mathrm{for}\ \ 0\leq f\in {\mathcal C}_0(V).
\end{equation}
In general, for any $f\in {\mathcal C}_0(V)$, write $f=f_+-f_-$ with $f_+=\max\{f,0\}$ and $f_-=\max\{-f,0\}$. Then $|Df|\le |Df_+|+|Df_-|$ and consequently $\||Df|\|_{\ell_p(E,\nu)} \le \||Df_+|\|_{\ell_p(E,\nu)}+\||Df_-|\|_{\ell_p(E,\nu)}$, where $f_+,f_-\in {\mathcal C}_0(V)$ are non-negative.

The argument below is essentially due to \cite[Theorem~1.3]{Dungey08} modified from the continuous setting.
For any $t>0, x\in V,$ set $u(t,x):=e^{-t\Delta}f(x)$ and $$J(t,x):=-(\partial_t+\Delta)(u^p(t,x)).$$ By
\eqref{eq:Gamma equ} and H\"older's inequality, we have
\begin{eqnarray}
\norm{\sqrt{\Gamma_p(u)}}_{\ell^p(V,\nu)}^p
&=&
\sum_{x} u(t,x)^{\frac{p(2-p)}{2}}J(t,x)^{\frac{p}{2}}\nu_x\nonumber\\
&\leq&
\left(\sum_{x}u(t,x)^p\nu_x\right)^{\frac{2-p}{2}}\left(\sum_{x}J(t,x)\nu_x\right)^{\frac{p}{2}}.\label{eq:equ1}
\end{eqnarray}
Since $u(t,\cdot)\in \ell^p(V,\nu)$ for any $t>0,$ Lemma \ref{lem:laplace zero} yields
$$
\sum_x\Delta (u^p)(t,x)\nu_x=0.$$
Hence
\begin{eqnarray*}
  \sum_{x}J(t,x)\nu_x
  &=&
  -\sum_{x}\partial_t(u^p(t,x))\nu_x=-p\sum_x u^{p-1}\partial_t u\nu_x\\
  &\leq&
  p\|u\|_{\ell^p(V,\nu)}^{p-1}\|\partial_tu\|_{\ell^p(V,\nu)}.
\end{eqnarray*}
Combining this with \eqref{eq:equ1}, we have
\begin{equation}\label{eq:equ3}
\norm{\sqrt{\Gamma_p(u)}}_{\ell^p(V,\nu)}^2 \leq C \|u\|_{\ell^p(V,\nu)}\|\Delta u\|_{\ell^p(V,\nu)}.
\end{equation}
Since $e^{-t\Delta}$ extends to a semigroup on $\ell^p(V,\nu)$ for $p\in (1,\infty),$
\begin{eqnarray*}
&&\lim_{t\to0^+}\|e^{-t\Delta}f\|_{\ell^p(V,\nu)}=\|f\|_{\ell^p},\\
&&\lim_{t\to0^+}\|\Delta e^{-t\Delta}f\|_{\ell^p(V,\nu)}=\lim_{t\to0^+}\|e^{-t\Delta}\Delta f\|_{\ell^p(V,\nu)}=\|\Delta f\|_{\ell^p},
\end{eqnarray*}
where we used the fact that $f\in {\mathcal C}_0(V)$ lies in the domain of the generator $\Delta$. Note that $\displaystyle\lim_{t\to0+}u(t,x)=f(x), \forall x\in V,$ Fatou's lemma yields
$$
\norm{\sqrt{\Gamma_p(f)}}_{\ell^p(V,\nu)}^2\leq \liminf_{t\to0^+}\norm{\sqrt{\Gamma_p(u)}}_{\ell^p(V,\nu)}^2.
$$
Hence, passing to the limit $t\to 0^+$ in \eqref{eq:equ3}, we get \eqref{eq:equ2}, and prove the theorem.
\end{proof}







\section{Littlewood-Paley-Stein functions}\label{LPS}
To prove the boundedness of Littlewood-Paley-Stein functions, we need the following lemma from semigroup theory, see \cite[Section~III.3]{Stein70}. Given $f\in \ell^1(V,\nu)+\ell^\infty(V,\nu),$ the semigroup maximal function of $f$ is defined as
$$f^*(x):=\sup_{t>0}|(e^{-t\Delta}f)(x)|,\quad x\in V.$$
\begin{lemma}\label{lem:semigroup lem}
  For any $p\in(1,\infty],$
  $$\|f^*\|_{\ell^p(V,\nu)}\leq C_p\|f\|_{\ell^p(V,\nu)},\quad \forall f\in \ell^p(V,\nu).$$
\end{lemma}

Using this lemma, one can prove the following.
\begin{lemma}\label{lem:derivative}
  Let $p\in(1,\infty).$ For any $0\leq f\in {\mathcal C}_0(V)$ and $t>0,$
\begin{eqnarray}\label{eq:equ4}
\partial_t\sum_{x\in V}(e^{-t\Delta}f(x))^p\nu_x=\sum_{x\in V}\partial_t((e^{-t\Delta}f(x))^{p})\nu_x.
\end{eqnarray}
\end{lemma}
\begin{proof}
A justification is needed since the summation here involves infinitely many terms. Set $u(t,x)=e^{-t\Delta}f(x).$ Then
 \begin{eqnarray*}
|\partial_t u^p(t,x)| &=& |(u(t,x))^{p-1}\Delta u(t,x)|
\\ &\leq& p\sum_{x\in V}\sup_{t>0}\{(u(t,x))^{p-1}|\Delta u(t,x)|\}\nu_x
\\ &:=&p\sum_{x\in V}g(x)\nu_x.
 \end{eqnarray*}
We claim that $g\in \ell^1(V,\nu).$ Indeed, note that
    \begin{eqnarray*}
      g(x)&=&\sup_{t>0}\{(e^{-t\Delta}f(x))^{p-1}|e^{-t\Delta}(\Delta f)(x)|\}\\
      &\leq&\sup_{t>0}\{(e^{-t\Delta}f(x))^{p-1}e^{-t\Delta}(|\Delta f|)(x)\}\\
      &\leq& (f^*(x))^{p-1} (|\Delta f|)^*(x).
    \end{eqnarray*}
 By Lemma \ref{lem:semigroup lem}, $f^*\in \ell^p(V,\nu)$ and $|\Delta f|\in \ell^p(V,\nu).$ Hence H\"older's inequality yields that $g\in \ell^1(V,\nu).$

 Now the differentiability theorem, derived from the dominated convergence theorem, yields \eqref{eq:equ4}.
\end{proof}


Let $p\in [1,\infty)$ and $T:{\mathcal C}_0(V)\to {\mathcal C}_0(V)$ be a linear map. We denote by $\|T\|_{p\rightarrow p}:=\sup\{\|Tf\|_{\ell^p(V,\nu)}: \|f\|_{\ell^p(V,\nu)}=1, f\in {\mathcal C}_0(V)\}$ the operator norm of $T$ from $\ell^p(V,\nu)$ to $\ell^p(V,\nu).$ The following lemma is standard, see e.g. \cite[Proof~of~Theorem~1.3]{CoulhonDuongLi03}.
\begin{lemma}\label{lem:operator norm}
  Let $(V,E,\nu,\mu)$ be a weighted graph with positive spectral gap, i.e. $\lambda_1=\inf\sigma(\Delta)>0.$ Then
  $$
  \|e^{-t\Delta}\|_{p\rightarrow p}\leq e^{-2\lambda_1(p-1)t}.
  $$
\end{lemma}
\begin{proof}
  Since $\|e^{-t\Delta}\|_{2\rightarrow2}=e^{-\lambda_1t}$ and $\|e^{-t\Delta}\|_{1\rightarrow 1}\leq 1,$ the lemma follows from the Riesz-Thorin interpolation theorem.
\end{proof}

To obtain the Littlewood-Paley-Stein estimate for graphs, we first introduce an intermediate quantity,
 defined by Dungey \cite{Dungey08}. For any $a\in\R,$ $0\leq f\in {\mathcal C}_0(V),$ define
 $$
 (\mathcal{H}_{p,a}f)(x):=\left(\int_0^\infty e^{at}\Gamma_p(e^{-t\Delta}f)(x)\right)^\frac{1}{2} dt,\quad x\in V.
 $$
Recall that
 $$
(\mathcal{H}_af)(\{x,y\}):=\left(\int_0^\infty e^{at} |D(e^{-t\Delta}f)|^2(\{x,y\})dt\right)^\frac{1}{2},\quad \{x,y\}\in E.
$$

The following  lemma compares these two quantities.
\begin{lemma}\label{lem:littlewood key}
   Let $G=(V,E,\nu,\mu)$ be a graph satisfying \eqref{e:bounded Laplacian}, $p\in(1,2]$ and $a\in\R$.
   For any $0\leq f\in {\mathcal C}_0(V),$ then
$$
\|\HH_af\|_{\ell^p(E,\mu)} \le C\|\HH_{p,a}f\|_{\ell^p(V,\nu)}.
$$
\end{lemma}
\begin{proof} Set $u(t,x)=e^{-t\Delta}f(x).$ Then by \eqref{e:bounded Laplacian}, H\"older's inequality and Lemma \ref{lem:sym}, we have
\begin{eqnarray*}
&&\|\HH_{p,a}f\|_{\ell^p(V,\nu)}^p\\&=&\sum_{x}\nu_x\left(\sum_{y}\frac{\mu_{xy}}{\nu_x}\int_0^\infty e^{at}\gamma_p(u(t,x),u(t,y))dt\right)^{\frac{p}{2}}
\\ &\geq&
C\sum_{x,y}\mu_{xy}\left(\int_0^\infty e^{at}\gamma_p(u(t,x),u(t,y))dt\right)^{\frac{p}{2}}
\\ &=&
\frac{C}{2}\sum_{x,y}\mu_{xy}\left\{\left(\int_0^\infty e^{at}\gamma_p(u(t,x),u(t,y))dt\right)^{\frac{p}{2}}+\left(\int_0^\infty e^{at}\gamma_p(u(t,y),u(t,x))dt\right)^{\frac{p}{2}}\right\}
\\ &\geq&
C\sum_{x,y}\mu_{xy}\left(\int_0^\infty e^{at}[\gamma_p(u(t,x),u(t,y))+\gamma_p(u(t,y),u(t,x))]dt\right)^{\frac{p}{2}}
\\ &\geq&
C\|\HH_af\|_{\ell^p(E,\mu)}^p.
\end{eqnarray*}
This proves the lemma.
\end{proof}

We are now ready to prove Theorem \ref{thm:main thm2}.
\begin{proof}[\bf Proof of Theorem \ref{thm:main thm2}] We only prove the second assertion since the first one follows from the same argument.

It is enough to prove \eqref{eq:alpha pleq} for $0\leq f\in {\mathcal C}_0(V)$. Indeed, for any $f\in {\mathcal C}_0(V)$, write $f=f_+-f_-$ with $f_+=\max\{f,0\}$ and $f_-=\max\{-f,0\}$. Then it follows that $(\mathcal{H}_af)(\{x,y\})\le (\mathcal{H}_af_+)(\{x,y\})+(\mathcal{H}_af_-)(\{x,y\})$ for any $\{x,y\}\in E.$

By Lemma \ref{lem:littlewood key}, it suffices to show that for any $t>0,$
  \begin{equation}\label{eq:equ5}\|\mathcal{H}_{p,a}f\|_{\ell^p(V,\nu)}\leq C\|f\|_{\ell^p(V,\nu)},\quad \forall\ 0\leq f\in {\mathcal C}_0(V).\end{equation} The following is a discrete version of Stein's argument as in \cite[Theorem~1.3]{CoulhonDuongLi03} and \cite[Theorem~1.3]{Dungey08}.

  Set $u(t,x)=e^{-t\Delta}f(x).$
  By \eqref{eq:Gamma equ}, we set $$J_a(x)=-\int_0^\infty e^{at}(\partial_t+\Delta)(u^p)dt\geq0.$$ Write
 $$
 (\HH_{p,a}f)^2(x)\leq C\sup_{t>0}u^{2-p}(t,x)J_a(x)=Cf^*(x)^{2-p}J_a(x).
 $$
  By H\"older inequality and Lemma \ref{lem:semigroup lem},
\begin{equation}\label{eq:equ6}
\|\mathcal{H}_{p,a}f\|^p_{\ell^p(V,\nu)}\leq C\|f\|_{\ell^p(V,\nu)}^{\frac{p(2-p)}{2}}\|J_a\|_{\ell^1(V,\nu)}^{\frac{p}{2}}.\end{equation}
By Lemma \ref{lem:laplace zero} and Lemma \ref{lem:derivative}
\begin{eqnarray*}
\|J_a\|_{\ell^1(V,\nu)}&=&-\sum_{x}\nu_x\int_0^\infty e^{at}(\partial_t+\Delta)(u^p)(t,x)dt\\
  &=&-\int_0^\infty e^{at}\sum_{x}\nu_x(\partial_t+\Delta)(u^p)(t,x)dt\\
  &=&-\int_0^\infty e^{at}\sum_{x}\nu_x\partial_t(u^p)(t,x)dt\\
  &=&-\int_0^\infty e^{at}\partial_t\sum_{x}\nu_xu^p(t,x)dt.
  \end{eqnarray*}
Hence by  integration by parts, Lemma \ref{lem:operator norm} and $a<2(p-1)\lambda_1$,
\begin{eqnarray*}
\|J_a\|_{\ell^1(V,\nu)}&=&-\int_0^\infty \partial_t\Big(e^{at}\sum_{x}\nu_xu^p(t,x)\Big)dt+a\int_0^\infty e^{at}\sum_{x}\nu_xu^pdt\\
&\leq&\sum_{x}\nu_xu^p(0,x)+a\int_0^\infty e^{at}\|e^{-t\Delta}f\|_{\ell^p(V,\nu)}^pdt\\
&\leq& \|f\|_{\ell^p(V,\nu)}^p\br{1+a\int_0^\infty e^{at}e^{-2\lambda_1(p-1)t}dt}\leq C\|f\|_{\ell^p(V,\nu)}^p.
\end{eqnarray*}
Combining this with \eqref{eq:equ6} yields \eqref{eq:equ5}  hence the theorem.
\end{proof}

\bigskip

\section{Bounded Laplacians and positive spectral gap}\label{s:Riesz transform}

We first recall a result of \cite{AuscherCoulhonDuongHofmann04}: a positive spectral gap implies the $L^p$ boundedness, $1<p<\infty$, of Riesz transforms for manifolds with local bounded geometry.
More precisely, let $M$ be a complete Riemannian manifold satisfying the following conditions:
\begin{enumerate}[(I)]
  \item $M$ has  exponential volume growth: for any $r_0>0,$ there exist $C_1,C_2$ depending only on $r_0$ such that for any $x\in M,\theta>1, r\leq r_0,$
  $$\vol(B_x(\theta r))\leq C_1e^{C_2\theta}\vol(B_x(r)),$$ where $\vol(B_x(r))$ denotes the Riemannian volume of the geodesic ball of radius $r$ centered at $x$.
  \item There exist constants $C_3,C_4$ such that for any $x,y\in M$ and $t\in(0,1],$
$$p_t(x,x)\leq \frac{C_3}{\vol(B_x(\sqrt t))}$$ and $$|\nabla_x p_t(x,y)|\leq \frac{C_4}{\sqrt{t}\ \vol(B_y(\sqrt t))}.$$
\end{enumerate}
Then the following local equivalence holds: for any $p\in(1,\infty)$
\begin{equation*}\label{e:boundedness}
C^{-1}\br{\|\Delta_M^{1/2}f\|_{L^p}+\|f\|_{L^p}}
\leq \| |\nabla f| \|_{L^p}+\|f\|_{L^p}\leq C\br{\|\Delta_M^{1/2}f\|_{L^p}+\|f\|_{L^p}},\quad \forall f\in {\mathcal C}_0^{\infty}(M).
\end{equation*}
 See  \cite[Theorem~1.7]{AuscherCoulhonDuongHofmann04}.

By  spectral theory, $\Delta_M^{-1/2}$ is a bounded operator on $L^p(M,\vol)$ for any $p\in [1,\infty)$ if $\Delta_M$ has a positive spectral gap, see \cite{CoulhonDuong99} or \cite{JiKunstmannWeber10}.
Combining these results, one figures out that for any $f\in {\mathcal C}_0^{\infty}(M)$
\begin{eqnarray*}
\||\nabla f|\|_{L^p}&\leq& C\br{\|\Delta_M^{1/2}f\|_{L^p}+\|f\|_{L^p}}\\
&=&
C\br{\|\Delta_M^{1/2}f\|_{L^p}+\|\Delta_M^{-1/2}(\Delta_M^{1/2}f)\|_{L^p}}
\leq C\|\Delta_M^{1/2}f\|_{L^p}.
\end{eqnarray*}

Let $(V,E,\deg,\mu)$ be a weighted graph with normalized Laplacian. Then  assumption $(I)$ implies the local doubling condition \eqref{a:local doubling}, which can be seen by choosing $\theta=2,r=1/2.$


Instead of assuming \eqref{a:local doubling}, we consider bounded Laplacians on graphs: let $G=(V,E,\nu,\mu)$ be a weighted graph satisfying \eqref{e:bounded Laplacian}. For any $\Omega\subset V,$ the boundary of $\Omega$ is defined as $\partial \Omega:=\{e\in E| e=\{x,y\},\ \mathrm{for\ some}\  x\in \Omega, y\in V\setminus \Omega\}.$ The volume of $\Omega$ w.r.t. the measure $\nu$ is defined as $\nu(\Omega):=\sum_{x\in \Omega}\nu_x$ and the volume of $\partial \Omega,$ as $\mu(\partial \Omega):=\sum_{e\in E}\mu_{e}.$
We say that $G=(V,E,\nu,\mu)$ satisfies the linear isoperimetric inequality if
$$\
h:=\inf_{\substack{\Omega\subset V\\ \sharp \Omega<\infty}}\frac{\mu(\partial\Omega)}{\nu(\Omega)}>0,
$$
where $h$ is called the Cheeger constant of the graph $G.$ As is well known, see e.g. Woess \cite[Proposition~4.3]{Woess00}, $G$ satisfies the linear isoperimetric inequality if and only if
the following Sobolev inequality holds
\begin{equation}\label{e:Sobolev}
\|f\|_{\ell^1(V,\nu)}\leq C\||Df|\|_{\ell^1(E,\mu)}, \ \ \ \ \forall f\in {\mathcal C}_0(V).
\end{equation}


By spectral theory, a graph $G=(V,E,\nu,\mu)$ has a positive spectral gap if and only if
\begin{equation}\label{e:l2 spectral gap}
\|f\|_{\ell^2(V,\nu)}\leq C\||Df|\|_{\ell^2(E,\mu)},\ \ \ \forall f\in {\mathcal C}_0(V).
\end{equation} By Cheeger's estimate, it turns out that \eqref{e:Sobolev} and \eqref{e:l2 spectral gap} are equivalent for graphs with bounded Laplacians, see \cite{DodziukKarp88}.
\begin{thm}\label{thm:one part}
Let $G=(V,E,\nu,\mu)$ be a graph satisfying \eqref{e:bounded Laplacian}. Suppose that $G$ has a positive spectral gap, then for any $p\in [1,\infty),$
$$
\||Df|\|_{\ell^p(E,\mu)}\approx \|f\|_{\ell^p(V,\nu)},\ \ \forall f\in {\mathcal C}_0(V).
$$
\end{thm}
\begin{proof}
By Lemma \ref{l:bounded operator}, it suffices to prove that for any $p\in [1,\infty)$
$$
\|f\|_{\ell^p(V,\nu)}\leq C_p\||Df|\|_{\ell^p(E,\mu)},\ \ \ \forall f\in {\mathcal C}_0(V).
$$
For any finite subset $\Omega\subset V,$ we denote by $\mathds{1}_{\Omega}$ the characteristic function of $\Omega,$ i.e.
\begin{equation*}
\mathds{1}_{\Omega}(x)
=\left\{\begin{array}{ll}1, & x\in \Omega, \\ 0, & x\not\in \Omega.\end{array}\right.
\end{equation*}
By plugging $f=\mathds{1}_{\Omega}$ into \eqref{e:l2 spectral gap}, we get that the Cheeger constant is positive, i.e. $h>0.$ Hence the Sobolev inequality \eqref{e:Sobolev} holds. For any $g\in {\mathcal C}_0(V),$ plugging $f(x)=g^p(x)$ into \eqref{e:Sobolev}, we get
\begin{eqnarray*}
\|g\|_{\ell^p(V,\nu)}^p&\leq&C\sum_{x,y}|g^p(x)-g^p(y)|\mu_{xy}
\\  &\leq &
C\sum_{x,y}(|g(x)|^{p-1}+|g(y)|^{p-1})|g(y)-g(x)|\mu_{xy}
\\  &\leq &
C \left(\sum_{x,y}|g(x)-g(y)|^p\mu_{xy}\right)^{\frac1p} \left(\sum_{x,y}(|g(x)|^{p-1}+|g(y)|^{p-1})^{p'}\mu_{xy}\right)^{\frac1{p'}}
\\ &\leq &C \left(\sum_{x,y}|g(x)-g(y)|^p\mu_{xy}\right)^{\frac1p} \left(\sum_{x,y}(|g(x)|^{p}+|g(y)|^{p})\mu_{xy}\right)^{\frac1{p'}},
\end{eqnarray*}
where we used the mean value inequality in the second inequality and H\"older inequality in the third one in which $p'$ is the H\"older conjugate of $p,$ i.e. $1/p+1/{p'}=1.$ By  symmetry and \eqref{e:bounded Laplacian},
$$
\sum_{x,y}(|g(x)|^{p}+|g(y)|^{p})\mu_{xy}=2\sum_{x,y}|g(x)|^p\mu_{xy} \leq C\|g\|_{\ell^p(V,\nu)}^p.
$$
Hence
$$
\|g\|_{\ell^p(V,\nu)}^p\leq C\|D g\|_{\ell^p(E,\mu)}\|g\|_{\ell^p(V,\nu)}^{\frac{p}{p'}}.
$$
This yields the result.
\end{proof}

The following lemma was first proved by \cite[Theorem~1.3]{CoulhonDuong99}, see also \cite[Lemma~2.2]{JiKunstmannWeber10}. The proof follows verbatim from their arguments, so we omit the details.
\begin{lemma}\label{thm:bounded inverse}
    Let $G=(V,E,\nu,\mu)$ be a graph satisfying \eqref{e:bounded Laplacian}. Suppose that $G$ has a positive spectral gap, then for any $p\in (1,\infty),$
    \begin{equation}\label{e:bounded half Laplacian}
      \|\Delta^{-1/2}f\|_{\ell^p(V,\nu)}\leq C_p\|f\|_{\ell^p(V,\nu)}, \ \ \ \forall f\in {\mathcal C}_0(V).
    \end{equation}
\end{lemma}

Now we can prove our main theorem in this section.

\begin{proof}[\bf Proof of Theorem~\ref{thm:main theorem}]
By Theorem \ref{thm:one part} and Lemma \ref{thm:bounded inverse}, it suffices to prove that
\begin{equation*}
\|\Delta^{1/2}f\|_{\ell^p(V,\nu)}\leq C\|f\|_{\ell^p(V,\nu)}, \ \ \ \forall f\in {\mathcal C}_0(V).
\end{equation*}
For $f\in {\mathcal C}_0(V),$ let $g=\Delta^{-1/2}f.$ Then by Lemma \ref{thm:bounded inverse}, $g\in
\ell^p(V,\nu)$ for any $p\in (1,\infty).$ Hence by  functional calculus and noting that $\Delta$ is bounded  on
$\ell^p$, $\Delta^{1/2}f=\Delta (\Delta^{-1/2}f).$
Hence by Lemma \ref{l:bounded operator} and Lemma \ref{thm:bounded inverse},
$$
\|\Delta^{1/2}f\|_{\ell^p(V,\nu)}=\|\Delta(\Delta^{-1/2}f)\|_{\ell^p(V,\nu)}\\
\leq  C\|\Delta^{-1/2}f\|_{\ell^p(V,\nu)}\leq C\|f\|_{\ell^p(V,\nu)}.
$$
This proves the theorem.
\end{proof}

Now we can prove a generalization of \cite[Theorem~1.1]{JiKunstmannWeber10} for graphs.
\
\begin{proof}[\bf Proof of Theorem \ref{thm:generalization}]
By  assumption, $\sigma(\Delta)\subset\{0\}\cup[a,\infty)$ for some $a>0.$ This implies that the total $\nu$-measure of the graph $\nu(V)$ is finite. Indeed, suppose $g$ is an $\ell^2(V,\nu)$ eigenfunction pertaining to the eigenvalue $0,$ that is, $g$ is a harmonic function. Then the $\ell^2$ Liouville theorem (see e.g. \cite{HuaJost13} or \cite{HuaKeller14}) yields that $g$ is a constant function. The fact  that $g\in \ell^2(V,\nu)$ yields that  $\nu(V)$ is finite.

  We define for any $p\in[1,\infty),$ $$\ell_0^p(V,\nu):=\left\{f\in \ell^p(V,\nu)| \sum_{x\in V}f(x)\nu_x=0\right\}.$$
  By the fact that $\Delta$ is a bounded operator, $\ell^p_0(V,\nu)$ is an invariant subspace for the semigroup $T_p(t):=e^{-t\Delta}:\ell^p(V,\nu)\to \ell^p(V,\nu)$ and $\ell^p(V,\nu)=\ell^p_0(V,\nu)\oplus\mathrm{span}\{1\}.$ Furthermore, $\Delta:\ell^p_0(V,\nu)\to \ell^p_0(V,\nu).$ In fact, for any $f\in \ell^p_0(V,\nu),$ the H\"older inequality implies that $f\in \ell^1(V,\nu).$ By Lemma \ref{l:bounded operator} and Lemma \ref{lem:laplace zero}, $\Delta f\in \ell^p_0(V,\nu).$ Setting $\Delta_0:=\Delta|_{\ell^p_0(V,\nu)},$ we have the decomposition of the Laplacian according to $\ell^p(V,\nu)=\ell^p_0(V,\nu)\oplus\mathrm{span}\{1\}$, $$\Delta=\Delta_0\oplus 0,$$ where $0$ is the null operator.

Since the graph is connected, the $\ell^2$ Liouville theorem for harmonic functions in \cite{HuaJost13, HuaKeller14} implies $\{0\}$ is an eigenvalue of multiplicity one, which yields
  that $\inf(\sigma(\Delta_0))>0.$ Applying the same argument in the proof of Theorem \ref{thm:main theorem} to the operator $\Delta_0$ on $\ell^p_0(V,\nu),$ we prove the theorem.
\end{proof}



\bibliographystyle{alpha}
\newcommand{\etalchar}[1]{$^{#1}$}

\end{document}